\documentclass[a4paper,10pt]{article}

\usepackage[ascii]{inputenc}
\usepackage{amsmath,amsfonts,amssymb,amsthm}
\usepackage{mathrsfs}
\usepackage{xcolor}

\newtheoremstyle{theorem}{5pt}{5pt}{\itshape}{}{\bfseries}{.}{.5em}{}

\theoremstyle{theorem}
\newtheorem{theorem}{Theorem}
\newtheorem{lemma}[theorem]{Lemma}
\newtheorem{corollary}[theorem]{Corollary}

\usepackage{titlesec}
\titleformat{\section}
{\normalfont\bfseries}{\large\thesection}{1em}{\large}
\titlespacing*{\section}{0pt}{3.5ex plus 1ex minus .2ex}{2.3ex plus .2ex}
\titleformat{\subsection}
{\normalfont\bfseries}{\thesubsection}{1em}{\normalsize}
\titlespacing*{\subsection}{0pt}{3.5ex plus 1ex minus .2ex}{2.3ex plus .2ex}

\allowdisplaybreaks[3]

\begin{document}

\title{On Sums Involving Fourier \\ Coefficients of Maass Forms for $\mathrm{SL}(3,\mathbb Z)$}
\author{Jesse J\"a\"asaari\footnote{Department of Mathematics and Statistics, University of Helsinki}{ } and Esa V\!. Vesalainen\footnote{BCAM --- Basque Center for Applied Mathematics}}
\date{}
\maketitle

\begin{abstract}
\noindent We derive a truncated Voronoi identity for rationally additively twisted sums of Fourier coefficients of Maass forms for $\mathrm{SL}(3,\mathbb Z)$, and as an application obtain a pointwise estimate and a second moment estimate for the sums in question.
\end{abstract}

\section{Introduction}

The purpose of this paper is to estimate sums
\[\sum_{m\leqslant x}A(m,1)\,e\!\left(\frac{mh}k\right)\]
where $A(m,1)$ are Fourier coefficients of a fixed Maass form for $\mathrm{SL}(3,\mathbb Z)$ and $h/k$ is a reduced fraction with a small denominator.
Estimating exponential sums is a central and an interesting problem in number theory. Classical circle and Dirichlet's divisor problems, as well as their twisted versions, have been under investigation for a long time and the analogous problem for estimating exponential sums weighted by Fourier coefficients of automorphic forms is natural as not much is known about these coefficients. The classical and cusp form problems also exhibit fascinating analogies, largely because the classical problems have modular origins as well. 

Arguably the first exponential sum result of this kind is Wilton's classical result \cite{Wilton}, which after Jutila's improvement \cite{Jutila1987} says that
\[\sum_{m\leqslant x}a(m)\,e(m\alpha)\ll x^{1/2}\]
uniformly for $\alpha\in\mathbb R$, where $a(m)$ are the normalized Fourier coefficients of a fixed holomorphic cusp form for the full modular group. The analogous result for Maass forms for $\mathrm{SL}(2,\mathbb Z)$ was established in \cite{Jaasaari--Vesalainen}.

As is usual with such problems, the behaviour of the sum in question depends heavily on whether $\alpha$ is near a fraction with a small denominator or not, the former case exhibiting the analogies to the classical problems. In addition, behaviour near such a fraction heavily depends on the behaviour at the fraction. In particular, one expects better estimates. Indeed, for $\alpha=h/k$ with small $k$, the above sum is of the order of magnitude $k^{1/2}\,x^{1/4}$ in the mean and certainly $\ll k^{2/3}\,x^{1/3+\varepsilon}$ pointwise \cite[Chapter 1]{Jutila}. For Maass forms for $\mathrm{SL}(2,\mathbb Z)$, similar results are true, but the pointwise estimate reflects the fact that the Ramanujan--Petersson conjecture is not known to hold for Maass forms. Indeed, one has Meurman's pointwise estimate $\ll k^{2/3}\,x^{1/3+\vartheta/3+\varepsilon}$ \cite{Meurman}, where $\vartheta$ is the exponent towards the Ramanujan--Petersson conjecture.

We are interested in the analogous sums for Maass forms for $\mathrm{SL}(3,\mathbb Z)$. In this setting, a result similar to the classical Wilton--Jutila estimate has been established by Miller \cite{Miller}: one has
\[\sum_{m\leqslant x}A(m,1)\,e(m\alpha)\ll x^{3/4+\varepsilon},\]
uniformly for $\alpha\in\mathbb R$. From this one also sees how moving from the $\mathrm{GL}(2)$ world to the $\mathrm{GL}(3)$ world poses extra challenges: one conjectures that the above sum is $\ll x^{1/2+\varepsilon}$.

In this paper we study the above sums for $\alpha=h/k$. The main tools used in the $\mathrm{GL}(2)$ setting are the truncated Voronoi identities for the sums in question. Thus, we establish such a truncated Voronoi identity for $\mathrm{GL}(3)$ Maass forms and use it to estimate the aforementioned sums pointwise and in the mean.

In the case of Maass forms the Ramanujan--Petersson conjecture is not currently known and so we are also interested in how our results depend on the growth rate of the Fourier coefficients. Namely if we have the bound of the form $A(m_1,m_2)\ll (m_1m_2)^{\vartheta+\varepsilon}$ for some $\vartheta\in[0,\infty[$ and for all $\varepsilon>0$, then we are interested to know how $\vartheta$ shows up in the estimates. 

\section{The main results}

\subsection{Truncated Voronoi identity}

Let us consider a fixed Maass form for $\mathrm{SL}(3,\mathbb Z)$ with the Fourier coefficients $A(m_1,m_2)$.
Before stating our first theorem, given a parameter $N\in\left[2,\infty\right[$, we define for each $k\in\mathbb Z_+$ with $k\leqslant N$ a positive integer $N_k$ so that, say $N\leqslant N_k\leqslant 2N$, and
\[\left\|\frac{N_k+1/2}{d^2}\right\|\gg\frac1{d\left(1+\log^2k\right)},\]
for every $d\mid k$ with $d\leqslant\sqrt{2N}$,
where $\left\|\cdot\right\|$ denotes the distance from the set of integers. Such an $N_k$ exists by Lemma \ref{parameter-modification} below. Why this is useful will become clear later. In any case, we have the following twisted truncated Voronoi identity.
\begin{theorem}\label{gl3-truncated}
Let $x,N\in\left[2,\infty\right[$ with $N\ll x$, and let $h$ and $k$ be coprime integers with $1\leqslant k\leqslant x$, $k\leqslant N$ and $k\ll(Nx)^{1/3}$, the latter having a sufficiently small implicit constant depending on the underlying Maass form. Then we have
\begin{align*}
&\sum_{m\leqslant x}A(m,1)\,e\!\left(\frac{mh}k\right)\\
&=\frac{x^{1/3}}{\pi\,\sqrt3}\sum_{d\mid k}\frac1d\sum_{d^2m\leqslant N_k}
\frac{A(d,m)}{m^{2/3}}\,S\!\left(\overline h,m;\frac kd\right)\cos\!\left(\frac{6\pi d^{2/3}(mx)^{1/3}}k\right)\\
&\qquad+O(k\,x^{2/3+\vartheta+\varepsilon}\,N^{-1/3})+O(k\,x^{1/6+\varepsilon}\,N^{1/6+\vartheta}).
\end{align*}
\end{theorem}
\noindent
Here the implicit constant in the condition $N\ll x$ can be chosen freely, but the one in the condition $k\ll(Nx)^{1/3}$ depends on the underlying Maass form. Furthermore, the implicit constants in the result depend on the Maass form and on $\varepsilon$, and $\vartheta\in\left[0,\infty\right[$ is such that $A(m_1,m_2)\ll (m_1m_2)^{\vartheta+\varepsilon}$. By the work of Kim and Sarnak \cite{Kim--Sarnak} we know that $\vartheta=5/14$ is admissible. The notation $\overline h$ denotes an arbitrary integer such that $h\overline h\equiv1\pmod k$, and $S(h,m;k)$ denotes the usual Kloosterman sum
\[S(h,m;k)=\sum_{x\in\mathbb Z_k^\times}e\!\left(\frac{hx+m\overline x}k\right),\]
where $\overline x$ denotes a residue class modulo $k$ such that $x\overline x\equiv1\pmod k$. 
We could replace here $N_k$ by $N$ but then the second error term would become the weaker $O(k^{2-\vartheta}\,x^{1/6+\varepsilon}\,N^{1/6+\vartheta})$. On the other hand, replacing $N_k$ by $N$ allows values $k>N$.

In the $\mathrm{GL}(2)$ setting, twisted truncated Voronoi identities have been derived previously by Jutila \cite{Jutila} for the error term in the classical Dirichlet divisor problem and for sums involving Fourier coefficients of holomorphic cusp forms, and by Meurman \cite{Meurman} for sums involving Fourier coefficients of $\mathrm{GL}(2)$ Maass forms. Truncated Voronoi identities have been key tools in the study of the classical circle and divisor problems (see e.g.\ \cite[Chapter 13]{Ivic}).

Truncated Voronoi identities are reasonably sharp approximate formulations of full Voronoi identities, which are essentially what one gets if one formally substitutes a characteristic function of an interval to the corresponding Voronoi summation formula. Naturally, this formal substitution is analytically challenging as the Voronoi summation formulae typically require the test functions to be smooth enough, and convergence problems get more difficult when moving to higher rank setting. Twisted Voronoi summation formulae have been implemented for $\mathrm{GL}(3)$ by Miller and Schmid \cite{Miller--Schmid} and for $\mathrm{GL}(n)$ in \cite{Goldfeld--Li1, Miller--Schmid2008}. A~truncated Voronoi identity for the generalized divisor problem involving $d_k(n)$ has been  given in \cite{Ivic--Zhai}, and a truncated Voronoi identity for plain sums of coefficients of fairly general $L$-functions has been given in \cite{Friedlander--Iwaniec}.

\subsection{Second moment estimate}

The truncated Voronoi identity stated above can be used to derive upper bounds for second moments. In particular, we get:
\begin{theorem}\label{meansquare} Let $h$ and $k$ be positive integers such that $(h,k)=1$. Then for $X\in\left[1,\infty\right[$ we have
\begin{align*}
&\int\limits_1^X\left|\sum_{m\leqslant x}A(m,1)\,e\!\left(\frac{mh}k\right)\right|^2\mathrm d x
\ll 
k^2\,X^{5/3+2\vartheta+\varepsilon}.
\end{align*}
\end{theorem}

\noindent The mean square behaviour of the error term of the Dirichlet's divisor problem was considered by Cram\'er \cite{Cramer} and Jutila has studied the average behaviours of the error term in the additively twisted divisor problem and the rationally additively twisted exponential sums attached to holomorphic cusp forms \cite{Jutila1985, Jutila}. Chandrasekharan and Narasimhan \cite{Chandrasekharan--Narasimhan1, Chandrasekharan--Narasimhan2} have established similar results for plain sums of coefficients of Dirichlet series of fairly general $L$-functions.

\subsection{Pointwise estimates}

The estimation of untwisted sums $\sum_{m\leqslant x}A(m,1,\ldots,1)$ with Fourier coefficients of a $\mathrm{GL}(n)$ Maass form has been considered before in the works \cite{Goldfeld--Sengupta, Lu, Meher}. In particular, the bound in \cite{Lu, Meher} is $\ll x^{(n^2-n)/(n^2+1)+\varepsilon}$. One also has the bound $\ll x^{(n-1)/(n+1)+\varepsilon}$ if $\vartheta=0$ \cite{Friedlander--Iwaniec}. In the case $n=2$ the best result to date seems to be due to \cite{Lu2}; the sum is $\ll x^{1027/2827+\varepsilon}$.

With rational additive twists, the sums seem to have been considered before only for $n=2$, for holomorphic cusp forms in \cite{Jutila} and for Maass forms in \cite{Meurman}. In fact, \cite{Meurman} considers also the dependence on the underlying Maass form. Of course, pointwise bounds for rational additive twists, and other special values such as suitable transcendental $\alpha$, have been considered before for $\mathrm{GL}(3)$ Maass forms in works such as \cite{Booker1, Booker2, Ren--Ye, Godber}, but only when wide weight functions are present, and such weighted sums behave quite differently from the unweighted sums.

For $\mathrm{GL}(3)$ Maass forms, we get estimates for the rationally additively twisted linear exponential sums as follows.
\begin{corollary}\label{pointwise-gl3}
Let $x\in\left[1,\infty\right[$, and let $h$ and $k$ be coprime integers with $1\leqslant k\ll x^{2/3}$. Then,
\[\sum_{m\leqslant x}A(m,1)\,e\!\left(\frac{mk}k\right)\ll k^{1/2+\varepsilon}\,x^{2/3}+k\,x^{1/3+\vartheta+\varepsilon}
.\]
Furthermore, when $\vartheta\leqslant1/3$ and $k\ll x^{2/3-2\vartheta}$, we have
\[\sum_{m\leqslant x}A(m,1)\,e\!\left(\frac{mh}k\right)\ll k^{3/4}\,x^{1/2+\vartheta/2+\varepsilon}+k^{9/8+3\vartheta/4}\,x^{1/4+3\vartheta^2/2+3\vartheta/4+\varepsilon}.\]
In particular, for $\vartheta=0$ and $k\ll x^{2/3}$, we have
\[\sum_{m\leqslant x}A(m,1)\,e\!\left(\frac{mh}k\right)\ll k^{3/4}\,x^{1/2+\varepsilon}.\]
\end{corollary}
Miller's \cite{Miller} bound gives $\ll x^{3/4+\varepsilon}$ for arbitrary real twists.  Under the Ramanujan--Petersson conjecture $\vartheta=0$, the above upper bound $\ll k^{3/4}\,x^{1/2+\varepsilon}$ is an improvement of $\ll x^{3/4+\varepsilon}$ when $k\ll x^{1/3}$. Let us mention in passing that the works \cite{Li--Young, Godber} have considered how Miller's bound depends on the underlying Maass form.

For comparison, for $\mathrm{GL}(2)$ Maass forms one has the bound $\ll k^{2/3}\,x^{1/3+\vartheta/3+\varepsilon}$ due to \cite{Meurman}. It was observed in \cite{Vesalainen} that for holomorphic cusp forms and larger $k$, improvements can be obtained using estimates for short exponential sums from \cite{Ernvall-Hytonen, Ernvall-Hytonen--Karppinen}. For $\mathrm{GL}(2)$ Maass forms analogous improvements were obtained in \cite{Jaasaari--Vesalainen}, for example the bound $\ll k^{(1-6\vartheta)/(4-6\vartheta)}\,x^{3/(8-12\vartheta)+\varepsilon}$ which holds for $x^{3/(5+6\vartheta)-1/2+\vartheta}\ll k\ll x^{5/18+\vartheta/3}$.

\section{Notation}

The symbols $\ll$, $\gg$, $\asymp$, $O$ and $o$ are used for the usual asymptotic notation: for complex valued functions $f$ and $g$ in some set $X$, the notation $f\ll g$ means that $\left|f(x)\right|\leqslant C\left|g(x)\right|$ for all $x\in X$ for some implicit constant $C\in\mathbb R_+$. When the implied constant depends on some parameters $\alpha,\beta,\ldots$, we use $\ll_{\alpha,\beta,\ldots}$ instead of mere $\ll$. The notation $g\gg f$ means $f\ll g$, and $f\asymp g$ means $f\ll g\ll f$.

All the implicit constants are allowed to depend on the underlying Maass form and on $\varepsilon$, which denotes an arbitrarily small fixed positive number, which may not the same on each occurrence.

As usual, complex variables are written in the form $s=\sigma+it$, and we write $e(x)$ for $e^{2\pi ix}$. The subscript in the integral $\int_{(\sigma)}$ means that we integrate over the vertical line $\Re s=\sigma$. For simplicity, we write $\left\langle\cdot\right\rangle$ for $(1+\left|\cdot\right|^2)^{1/2}$.

\section{Properties of the additively twisted $\mathrm{GL}(3)$ $L$-functions}

The additive twisted $\mathrm{GL}(3)$ formula will be derived by considering the additively twisted version of the Godement--Jacquet $L$-function attached to a Maass form $\psi$ for $\mathrm{SL}(3,\mathbb Z)$ with Fourier coefficients $A(m_1,m_2)$. For practical purposes, it is better to consider the Dirichlet series
\[L_j\!\left(s+j,\frac hk\right)=\sum_{m=1}^\infty\frac{A(m,1)}{m^{s}}\left(e\!\left(\frac{mh}k\right)+(-1)^j\,e\!\left(\frac{-mh}k\right)\!\right)\]
for $j\in\{0,1\}$.
This converges absolutely for $\sigma>1$ by the Rankin--Selberg estimate
\[\sum_{m\leqslant x}\left|A(m,n)\right|^2\ll_nx.\]
(For this, see e.g. \cite{Goldfeld}, Remark 12.1.8.)
This Dirichlet series has an entire analytic continuation and satisfies the functional equation
\[L_j\!\left(s+j,\frac hk\right)
=i^{-j}\,k^{-3s+1}\,\pi^{3s-3/2}\,G_j(s+j)\,\widetilde L_j\!\left(1-j-s,\frac{\overline h}k\right),\]
where $G_j(s)$ is given by the $\Gamma$-factors
\[G_j(s+j)=\frac{\displaystyle{\Gamma\!\left(\frac{1-s+j+\alpha}2\right)\Gamma\!\left(\frac{1-s+j+\beta}2\right)\Gamma\!\left(\frac{1-s+j+\gamma}2\right)}}{\displaystyle{\Gamma\!\left(\frac{s+j-\alpha}2\right)\Gamma\!\left(\frac{s+j-\beta}2\right)\Gamma\!\left(\frac{s+j-\gamma}2\right)}},\]
where $\alpha$, $\beta$ and $\gamma$ are
\[\alpha=-\nu_1-2\nu_2+1,\quad\beta=-\nu_1+\nu_2,\quad\text{and}\quad\gamma=2\nu_1+\nu_2-1,\]
where $\nu_1$ and $\nu_2$ and the spectral parameters of the underlying Maass form,
and where $\widetilde L(s,\overline h/k)$ is the Dirichlet series
\[\widetilde L_j\!\left(s-j,\frac{\overline h}k\right)
=\sum_{d\mid k}\sum_{m=1}^\infty\frac{A(d,m)}{d^{2s-1}\,m^{s}}\left(S\!\left(\overline h,m;\frac kd\right)+(-1)^j\,S\!\left(\overline h,-m;\frac kd\right)\!\right).\]
(See \cite{Goldfeld--Li1}, Section 3.) We remark that this series converges absolutely for $\sigma>1$. We recall Weil's bound for Kloosterman sums which says that
\[\left|S(h,m;k)\right|\leqslant d(k)\,(h,m,k)^{1/2}\,k^{1/2}.\]
Since in our applications of this, $h$ and $k$ will be coprime, the upper bound will always be $\ll d(k)\,k^{1/2}$.

We also recall the Rankin--Selberg estimate (see Chapter 12 of \cite{Goldfeld})
\[\sum_{d^2m\leqslant x}\left|A(d,m)\right|^2\asymp x,\]
so that
\[\sum_{m\leqslant x}\left|A(d,m)\right|^2\ll d^2x,
\quad\text{and}\quad\sum_{m\leqslant x}\left|A(d,m)\right|\ll dx.\]
In particular, for $\delta\in\mathbb R_+$,
\[\sum_{m=1}^\infty\frac{\left|A(d,m)\right|}{m^{1+\delta}}\ll_\delta d.\]
By using this and the Weil's bound, on the line $\sigma=1+\delta$, the function $\widetilde L_j(s-j,h/k)$ is
\begin{align*}
&\ll\sum_{d\mid k}d^{-1-2\delta}\sum_{m=1}^\infty\frac{\left|A(d,m)\right|}{m^{1+\delta}}\cdot d(k)\,k^{1/2}\,d^{-1/2}\\
&\ll\sum_{d\mid k}d^{-1-2\delta}\,d\,d(k)\,k^{1/2}\,d^{-1/2}\\
&\ll d(k)\,k^{1/2}\,\sum_{d\mid k}d^{-1/2-2\delta}\ll k^{1/2+\varepsilon}.
\end{align*}
An elementary application of Stirling's formula says that on the line $\sigma=-\delta$,
\[G_j(s+j)\ll\left\langle t\right\rangle^{3/2+3\delta},\]
where, in particular, the estimate is independent of $k$ as that parameter does not appear in $G_j(s+j)$. Finally, the power of $k$ in the functional equation for $L_j(s+j,h/k)$ is $\ll k^{1+3\delta}$. Thus, we see that on the line $\sigma=-\delta$, the function $L_j(s+j,h/k)$ is
\[\ll_\delta k^{3/2+3\delta+\varepsilon}\,t^{3/2+3\delta},\]
whereas on the line $\sigma=1+\delta$ it is $\ll_\delta1$. Thus, by the Phragm\'en--Lindel\"of -principle, we have
\[L_j\!\left(s+j,\frac hk\right)
\ll k^{3(1+\delta-\sigma)(1+\varepsilon)/2}\left\langle t\right\rangle^{3(1+\delta-\sigma)/2},\]
in the strip $-\delta\leqslant\sigma\leqslant1+\delta$.

Another elementary application of Stirling's formula says, that when $s$ lies in the vertical strips below and has sufficiently large imaginary part, the multiple $\Gamma$-factors can be replaced by a single quotient of two $\Gamma$-factors:
\[G_j(s+j)
=3^{3s-3/2}\frac{\Gamma\!\left(\frac{1-3s}2\right)}{\Gamma\!\left(\frac{3s-2}2\right)}\left(1+O(\left|s\right|^{-1})\right).\]

\section{Useful lemmas}

We shall extract the sum under study from the relevant $L$-function via a truncated Perron formula. The following version is Lemma 1.4.2 in~\cite{Brudern}.
\begin{lemma}\label{perron-formula}
Let $\sigma\in\mathbb R_+$ and let $c\colon\mathbb Z_+\longrightarrow\mathbb C$ be a sequence such that the Dirichlet series $\sum_{n=1}^\infty c(n)/n^\sigma$ converges absolutely. Then, for $x,T\in\left[2,\infty\right[$, we have
\[\sum_{n\leqslant x}c(n)
=\frac1{2\pi i}\int\limits_{\sigma-iT}^{\sigma+iT}\left(\sum_{n=1}^\infty\frac{c(n)}{n^s}\right)x^s\,\frac{\mathrm ds}s
+\text{error},\]
where the error is
\[\ll\frac{x^\sigma}T\sum_{n=1}^\infty\frac{\left|c(n)\right|}{n^\sigma}
+\left(1+\frac{x\log x}T\right)\max_{\frac{3x}4\leqslant n\leqslant\frac{5x}4}\left|c(n)\right|.\]
\end{lemma}

After Perron's formula, we will apply the relevant functional equation the $L$-function satisfies, and this will lead to integrals involving $\Gamma$-factors. Thus, we shall have plenty of opportunities to apply the classical Stirling formula, which can be found from e.g.\ Chapter 1 of \cite{Lebedev}.
\begin{lemma}
Let $\delta\in\left]0,\pi\right[$ and $R\in\mathbb R_+$ be fixed. Then
\[\Gamma(s)=\sqrt{2\pi}\,\exp\!\left(\left(s-\frac12\right)\log s-s\right)\left(1+O_{\delta,R}\!\left(\frac1{\left|s\right|}\right)\right)\]
for all $s\in\mathbb C$ with $\left|s\right|\geqslant R$ and $\left|\arg s\right|\leqslant\pi-\delta$. Furthermore, if we are given fixed real numbers $A,B\in\mathbb R$ with $A<B$, as well as a fixed number $T\in\mathbb R_+$, then we have
\[\Gamma(s)=\sqrt{2\pi}\,t^{s-1/2}\,\exp\!\left(-\frac{\pi t}2-it+\frac{\pi i}2\left(\sigma-\frac12\right)\right)\left(1+O_{A,B,T}(t^{-1})\right),\]
and
\[\left|\Gamma(s)\right|=\sqrt{2\pi}\,t^{\sigma-1/2}\,e^{-\pi t/2}\left(1+O_{A,B,T}(t^{-1})\right)\]
for all complex numbers $s$ with $A\leqslant\sigma\leqslant B$ and $t\geqslant T$.
\end{lemma}

The main terms of the truncated formulae will ultimately arise from the main terms of the asymptotics of the $J$-Bessel function. Thus, we need the following result, which can be found from e.g.\ Section 5.11 of \cite{Lebedev}.
\begin{lemma}
Let $\nu\in\mathbb R$ and $\delta\in\mathbb R_+$ be fixed. Then
\[J_\nu(x)=\sqrt{\frac2{\pi x}}\cos\!\left(x-\frac{\pi\nu}2-\frac\pi4\right)+O_{\nu,\delta}(x^{-3/2})\]
for $x\in\left[\delta,\infty\right[$.
\end{lemma}

Finally, we state for completeness the classical first derivative test, which is e.g.\ Lemma 4.3 in \cite{Titchmarsh}.
\begin{lemma}
Let $\left[a,b\right]$ be an interval of $\mathbb R$, and let $f,g\colon\left[a,b\right]\longrightarrow\mathbb R$ be such that $f$ is continuously differentiable on $\left[a,b\right]$ and that $g/f'$ is monotonic. Furthermore, let $M\in\mathbb R_+$ be a constant such that $f'(x)/g(x)\geqslant M$ for $x\in\left[a,b\right]$, or such that $f'(x)/g(x)\leqslant -M$ for $x\in\left[a,b\right]$. Then
\[\left|\int\limits_a^bg(x)\,e(f(x))\,\mathrm dx\right|\leqslant\frac2{\pi M}.\]
\end{lemma}

\section{Approximating some $\Gamma$-function integrals}

In the course of the derivation of the truncated Voronoi identities we will come across certain complex line integrals involving $\Gamma$-functions. In the following series of lemmas we shall express them approximately in terms of $J$-Bessel functions, and use this to obtain asymptotic information.

The computations are performed largely in the same spirit as those in Section 5 of \cite{Ernvall-Hytonen--Jaasaari--Vesalainen} and Section 1.4 of \cite{Jutila}. We shall define the more general integrals
\[\Omega_{\nu,k}(y;\delta,T)=\frac1{2\pi i}\int\limits_{-\delta-iT}^{-\delta+iT}\frac{\Gamma\!\left(\frac{1-ns}2\right)}{\Gamma\!\left(\frac{ns+1}2+\nu-\frac n2\right)}\left(s+\Lambda\right)^{-k}y^s\,\mathrm ds,\]
where integration is along a straight line segment, and where $\nu$ and $k$ are nonnegative integers, and $y$ and $T$ are positive real numbers. The parameter $\Lambda$ is a large positive real number, which will depend on $n$ and the underlying Maass form. The parameter $\delta$ will be a sufficiently small positive real constant. All the implicit constants in the following are most definitely going to depend on $n$, $\delta$, $\Lambda$, $k$ and $\nu$.

\begin{lemma}\label{obtaining-bessel-functions}
Let $\nu$ and $k$ be nonnegative integers, and let $y,T\in\left[1,\infty\right[$ with $y<(nT/2)^n$. Then
\begin{multline*}
\Omega_{\nu,k}(y;\delta,T)=\left(\frac n2\right)^{k-1}y^{1/2+(1-\nu-k)/n}\,J_{\nu+k-n/2}(2y^{1/n})\\+O(T^{n/2-\nu-k+n\delta})+O\!\left(T^{n/2-\nu-k}\frac1{\log\frac{n^nT^n}{2^ny}}\right).
\end{multline*}
\end{lemma}

\begin{proof}
Let us begin by replacing the factor $(s+\Lambda)^{-k}$ by a more complicated yet more suitable one via the identity
\begin{multline*}
\frac1{(s+\Lambda)^k}\\
=\frac{\left(\frac n2\right)^k}{\left(\frac{ns+1}2+\nu-\frac n2\right)
\left(\frac{ns+1}2+\nu+1-\frac n2\right)\cdots\left(\frac{ns+1}2+\nu+(k-1)-\frac n2\right)}\\
+O(\left\langle t\right\rangle^{-k-1}),
\end{multline*}
which holds for $s\in\mathbb C$ on the line $\sigma=-\delta$.
The integral involving $O(\left\langle t\right\rangle^{-k-1})$ can be estimated as
\[y^{-\delta}\int\limits_{-T}^T\left\langle t\right\rangle^{n\delta-\nu+n/2-k-1}\mathrm dt\ll y^{-\delta}\,\left(1+T^{n/2-\nu-k+n\delta}\right),\]
and so, by using the identity $\Gamma(s+1)=s\,\Gamma(s)$, we are left with the simpler integral
\[\left(\frac n2\right)^k\cdot\frac1{2\pi i}\int\limits_{-\delta-iT}^{-\delta+iT}
\frac{\Gamma\!\left(\frac{1-ns}2\right)}{\Gamma\!\left(\frac{ns+1}2+\nu+k-\frac n2\right)}\,y^s\,\mathrm ds.\]

Fundamentally, the idea now is to shift the line of integration to the right, and the Bessel function will then arise from the series of residues. The integrand has simple poles at the points $1/n$, $3/n$, $5/n$, \dots, except possibly finitely many of the first terms in which the possible poles of the denumerator cancel the corresponding poles of the numerator. However, in any case, the following calculations will give the same $J$-Bessel function expression (cf. e.g. Sect. 5.3 in \cite{Lebedev}).

The residue at the simple pole $(2j+1)/n$, where $j\in\mathbb Z_+\cup\left\{0\right\}$ is
\[\left(\frac n2\right)^k\left(-\frac2n\right)\cdot\frac{(-1)^j}{j!}\cdot\frac{y^{(2j+1)/n}}{\Gamma\!\left(j+1+\nu+k-\frac n2\right)}.\]
The series of these is
\begin{align*}
&=\left(\frac n2\right)^k\sum_{j=0}^\infty\left(-\frac2n\right)\cdot\frac{(-1)^j}{j!}\cdot\frac{y^{(2j+1)/n}}{\Gamma\!\left(j+1+\nu+k-\frac n2\right)}\\
&=-\left(\frac n2\right)^{k-1}\,y^{1/2+1/n-(\nu+k)/n}\sum_{j=0}^\infty\frac{(-1)^j}{j!}\cdot\frac{(2y^{1/n}/2)^{2j+\nu+k-n/2}}{\Gamma\!\left(j+1+\nu+k-\frac n2\right)}\\
&=-\left(\frac n2\right)^{k-1}\,y^{1/2+(1-\nu-k)/n}\,J_{\nu+k-n/2}(2y^{1/n}).
\end{align*}
Since we are integrating to the negative direction around the residues, the original integral is approximated by the same expression except for an opposing sign, provided that the shifts of the line of integration can be executed with a tolerable error.

We first shift the line of integration from the line segment connecting $-\delta-iT$ to $-\delta+iT$ to the line segment connecting $2N_0/n-iT$ to $2N_0/n+iT$, where $N_0\geqslant2$ is an arbitrarily large fixed positive integer. Integrals over the horizontal sides of the resulting rectangle are then estimated by absolute values to be
\[\ll_{N_0} y^{-\delta}\,T^{n/2-\nu-k+n\delta}+y^{2N_0/n}\,T^{n/2-\nu-k-2N_0}
\ll_{N_0} T^{n/2-\nu-k+n\delta}.\]

Next we wish to complete the integral to be over the whole vertical line $\sigma=2N_0/n$. To estimate the integrals over the vertical half-lines connecting $\sigma\pm iT$ to $\sigma\pm i\infty$, we apply Stirling's formula to get the approximation
\begin{multline*}
\frac{\Gamma\!\left(\frac{1-ns}2\right)y^s}{\Gamma\!\left(\frac{ns+1}2+\nu+k-\frac n2\right)}\\
=y^\sigma\left(\frac{nt}2\right)^{n/2-\nu-k-n\sigma}
\exp\!\left(-it\log\frac{n^nt^n}{2^ny}+itn+\frac{\pi i\left(n-2\nu-2k\right)}4\right)\\
\cdot\left(1+O_{N_0}(\left|s\right|^{-1})\right).
\end{multline*}
In particular, the integral $\int_{(2N_0/n)}$ is absolutely convergent.
The $t$-derivative of the phase here is
\[\frac\partial{\partial t}\left(-t\log\frac{n^nt^n}{2^ny}+tn\right)=-\log\frac{n^nt^n}{2^ny}.\]
Completing the integral to be over the entire vertical line $\sigma=2N_0/n$ yields two errors, the $O$-term of the asymptotics of the integrand contributing
\[\ll_{N_0}\int\limits_N^\infty y^{2N_0/n}\left\langle t\right\rangle^{n/2-\nu-k-2N_0-1}\,\mathrm dt\ll_{N_0} y^{2N_0/n}\,T^{n/2-n-k-2N_0}\ll_{N_0} T^{n/2-n-k},\]
and the main term contributing by the first derivative test
\[\ll_{N_0} y^{2N_0/n}\,T^{n/2-\nu-k-2N_0}\,\frac1{\log\frac{n^nT^n}{2^ny}}
\ll_{N_0} T^{n/2-\nu-k}\,\frac1{\log\frac{n^nT^n}{2^ny}}.\]

We shift the line of integration to the vertical line $\sigma=2N/n$, where $N\geqslant N_0$ is an arbitrarily large, but momentarily fixed positive integer larger than $n/2-\nu-k+3/2$. The integral over the horizontal line segment connecting points $2N_0/n+iU$ and $2N/n+iU$ is easily estimated to be
\[\ll_{N_0,N}y^{2N_0/n}\,U^{n/2-\nu-k-nN_0}+y^{2N/n}\,U^{n/2-\nu-k-nN},\]
which vanishes in the limit $U\longrightarrow\infty$. The same estimate applies for the integral over the horizontal line segment connecting $2N_0/n-iU$ and $2N/n-iU$. Furthermore, the integrals over the vertical lines involved are absolutely convergent, and therefore moving the line of integration has been justified.

Finally, let us estimate the integral over the vertical line $\sigma=2N/n$, and show that this vanishes in the limit $N\longrightarrow\infty$. Using repeatedly the functional equation $s\,\Gamma(s)=\Gamma(s+1)$, we can relate the $\Gamma$-factor of the denominator to $\Gamma$-factors on the line $\sigma=1/2$, leading to
\[\Gamma\!\left(\frac12-N-\frac{nit}2\right)\ll\frac{e^{\pi nt/4}}{(N-1)!}.\]
Similarly, we may relate the $\Gamma$-factor of the denominator to $\Gamma$-values on one of the lines $\sigma=1/2$ and $\sigma=1$, depending on the parity $n$, leading to an estimate of the shape
\[\frac1{\Gamma\!\left(N+\frac{nit}2+\nu+k-\frac n2+\frac12\right)}\ll\frac{e^{-\pi nt/4}}{\left|\frac12+\frac{int}2\right|^{N+\nu+k-n/2-1/2}},\]
where we emphasize that the implicit constant is independent of $N$ and $t$, even though it does depend on $n$, $\nu$ and $k$.
Combining the above estimates gives
\[\int\limits_{(2N/n)}\ldots\ll\frac{y^{2N/n}\,2^{N+\nu+k-n/2}}{(N-1)!}\int\limits_{-\infty}^\infty\frac{\mathrm dt}{\left\langle t\right\rangle^{N+\nu+k-n/2}}\longrightarrow0\]
as $N\longrightarrow\infty$, and we are done.
\end{proof}

Using the asymptotics of $J$-Bessel functions for $y\gg1$, we get the following following corollary in the special case $n=3$ for $k=1$ and $\nu=0$.
\begin{corollary}\label{corollary-on-gl3-integrals}
Let $y,T\in\left[1,\infty\right[$ with $y<(3T/2)^3$. Then
\begin{multline*}
\frac1{2\pi i}\int\limits_{-\delta-iT}^{-\delta+iT}
\frac{\Gamma\!\left(\frac{1-3s}2\right)}{\Gamma\!\left(\frac{3s-2}2\right)}\,\frac{y^s\,\mathrm ds}{s+\Lambda}
=\frac1{\sqrt\pi}\,y^{1/3}\,\cos\!\left(2y^{1/n}\right)\\
+O(T^{1/2+3\delta/2})+O\!\left(T^{1/2}\,\frac1{\log\frac{3^3T^3}{2^3y}}\right).
\end{multline*}
\end{corollary}

\section{Proof of the truncated $\mathrm{GL}(3)$ identity of Theorem \ref{gl3-truncated}}

\begin{lemma}\label{parameter-modification}
Let $k\in\mathbb Z_+$ and $N\in\left[2,\infty\right[$ with $k\leqslant N$. Then there exists $N_k\in\left[N,2N\right]\cap\mathbb Z$ such that
\[\left\|\frac{N_k+1/2}{d^2}\right\|\gg\frac1{d\left(1+\log^2k\right)}\]
for each $d\mid k$ with $d\leqslant\sqrt{2N}$. Here the implicit constant is of course independent of $N$, $k$ and $d$.
\end{lemma}

\begin{proof}
Since we always have $\left\|\ldots\right\|\geqslant1/(2d^2)$, it is clear that the condition always holds for $d\ll1$, and therefore also for $k\ll1$, and so we may assume that $k$ is, say, $\geqslant3$.
In the same vein, for $d\ll\log^2k$, we clearly have $\left\|\ldots\right\|\geqslant1/(2d^2)\gg1/(d\log^2k)$, and so we may restrict to values $d\gg\log^2k$, and a fortiori also to values $N\gg\log^2k$, where both implicit constants may be chosen as large as necessary.

Let $\nu$ be the number of integers in the interval $\left[N,2N\right]$. We have a priori $\nu$ acceptable choices of $N_k$, and we wish to weed out the unsuitable values. For a given large $d$ we remove from consideration those integers from the interval $\left[N,2N\right]$ whose absolutely least remainders modulo $d^2$ have absolute value smaller than, say, $d/{(20\log^2k)}$. The interval contains contains $\leqslant\nu/d^2$ complete residue systems modulo $d^2$ consisting of consecutive integers, and possibly parts of two other such systems. From each such system we are deleting $\leqslant d/(10\log^2k)+1$ numbers. Therefore we are deleting
\[\leqslant\nu\left(\frac1{d^2}+\frac2\nu\right)\left(\frac d{10\log^2k}+1\right)
\leqslant\nu\left(\frac1{d^2}+\frac8{d^2}\right)\cdot\frac d{9\log^2k}\leqslant\frac{\nu}{d\log^2k}\]
a priori possible choices, provided that $N/\log^2k$ and $d/\log^2k$ are large enough.

In total, we are removing
\[\leqslant\nu\sum_{\substack{d\mid k\\\log^2 k\ll d\leqslant\sqrt{2N}}}\frac1{d\log^2k}\ll\frac\nu{\log k}
\]
suitable values, and since $\nu\geqslant\left\lfloor N\right\rfloor$, this is certainly, say, $\leqslant N/2$ for large enough $N$ and $k$, and we are done.
\end{proof}

\begin{proof}[Proof of Theorem \ref{gl3-truncated}]
The proof is modelled after the proof of Theorem 1 in \cite{Jutila}. The existence of a suitable number $N_k$ has been established in Lemma \ref{parameter-modification} above. Let $x,T\in\left[2,\infty\right[$ with $T\ll x$ and $T\geqslant\Lambda$, where $\Lambda$ is larger than the absolute values of $\alpha$, $\beta$, $\gamma$ and~$1$.
Let us be given an $\varepsilon\in\mathbb R_+$ which will stay fixed during the proof. Fix an arbitrarily and sufficiently small $\delta\in\mathbb R_+$. We shall elaborate the meaning of ``sufficiently small'' in the course of the proof. By the truncated Perron formula of Lemma \ref{perron-formula},
\begin{multline*}
\sum_{m\leqslant x}A(m,1)\left(e\!\left(\frac{mh}k\right)+(-1)^j\,e\!\left(\frac{-mh}k\right)\!\right)\\
=\frac1{2\pi i}\int\limits_{1+\delta-iT}^{1+\delta+iT}L_j\!\left(s+j,\frac hk\right)x^s\,\frac{\mathrm ds}s
+O(x^{1+\vartheta+\varepsilon}\,T^{-1}),
\end{multline*}
provided that $\delta\leqslant\varepsilon$. We shall first derive a slightly more complicated truncated Voronoi identity from this. In the end the identity of the theorem statement will follow by simply averaging over $j\in\left\{0,1\right\}$.

We shift the line segment of integration from $\sigma=1+\delta$ to $\sigma=-\delta$. By the convexity estimate for the additively twisted $L$-function, and picking the pole of the integrand at the origin, the error from this shift is
\begin{align*}
&\ll
\int_{-\delta}^{1+\delta}x^{\sigma}\,T^{-1}\,k^{3(1+\delta-\sigma)(1+\varepsilon)/2}\,T^{3/2(1+\delta-\sigma)}\,\mathrm d\sigma+\left|L_j\!\left(0+j,\frac hk\right)\right|\\
\\ & \ll x^{1+\delta}\,T^{-1}+x^{-\delta}\,T^{1/2+3\delta}\,k^{3/2+4\delta}+k^{3/2+\varepsilon}\\
&\ll x^{1+\varepsilon}\,T^{-1}+k^{3/2}\,x^\varepsilon\,T^{1/2}+k^{3/2+\varepsilon},
\end{align*}
provided that $8\delta\leqslant\varepsilon$,
Furthermore, since we intend to apply Stirling's formula, we write
\begin{multline*}
\frac1{2\pi i}\int\limits_{-\delta-iT}^{-\delta+iT}L_j\!\left(s+j,\frac hk\right)\,x^s\,\frac{\mathrm ds}s\\
=\frac1{2\pi i}\left(\,\,\int\limits_{-\delta-iT}^{-\delta-i\Lambda}+\int\limits_{-\delta+i\Lambda}^{-\delta+iT}\,\,\right)L_j\!\left(s+j,\frac hk\right)\,x^s\,\frac{\mathrm ds}s+O(k^{3/2+\varepsilon}).
\end{multline*}

Now we may apply the functional equation of Godement--Jacquet $L$-functions, interchange the order of integration and summation and apply Stirling's formula to get
\begin{align*}
&\frac1{2\pi i}\left(\int+\int\right)L_j\!\left(s+j,\frac hk\right)\,x^s\,\frac{\mathrm ds}s\\
&=\frac1{2\pi i}\left(\int+\int\right)i^{-j}\,k^{-3s+1}\,\pi^{3s-3/2}\,G_j(s+j)\,\widetilde L_j\!\left(1-s-j,\frac{\overline h}k\right)\,x^s\,\frac{\mathrm ds}s\\
&=\frac1{2\pi i^j}\sum_{d\mid k}\frac1d\sum_{m=1}^\infty\frac{A(d,m)}m\left(S\!\left(\overline h,m;\frac kd\right)+(-1)^j\,S\!\left(\overline h,-m;\frac kd\right)\!\right)\\
&\qquad\cdot\left(\int+\int\right)\frac1i\,\,k^{-3s+1}\,d^{2s}\,\pi^{3s-3/2}\,3^{3s-3/2}\,\frac{\Gamma\!\left(\frac{1-3s}2\right)}{\Gamma\!\left(\frac{3s-2}2\right)}\\
&\qquad\qquad\cdot\left(1+O(\left|s\right|^{-1})\right)\,m^s\,x^s\,\frac{\mathrm ds}s.
\end{align*}
By Stirling's formula, in the region of integration the quotient of $\Gamma$-factors is $\ll t^{3/2+3\delta}$, and so the series corresponding to the $O$-term can be forgotten with an error
\begin{multline*}
\ll k^{1+3\delta}\sum_{d\mid k}\frac1{d^{1+2\delta}}\sum_{m=1}^\infty\frac{\left|A(d,m)\right|}{m^{1+\delta}}\,k^{1/2+\delta}
\int\limits_\Lambda^Tt^{3/2+3\delta}\,t^{-1}\,x^{-\delta}\,\frac{\mathrm dt}t
\ll k^{3/2}\,x^\varepsilon\,T^{1/2},
\end{multline*}
provided that $8\delta\leqslant\varepsilon$. Similarly, we may replace $1/s$ by $1/(s+\delta)$ with the same error $\ll k^{3/2}\,x^\varepsilon\,T^{1/2}$.

We treat the terms with $d^2m>N_k$ first. We are going to simplify the integral $\int+\int$ further by rewriting the complex line integrals as
\begin{multline*}
2\Re\int\limits_\Lambda^T
k^{-3(-\delta+it)+1}\,d^{2(-\delta+it)}\,(3\pi)^{3(-\delta+it)-3/2}\cdot\frac{\Gamma\!\left(\frac{1-3(-\delta+it)}2\right)}{\Gamma\!\left(\frac{3(-\delta+it)-2}2\right)}\left(mx\right)^{-\delta+it}\frac{\mathrm dt}{it}.
\end{multline*}

By Stirling's formula,
\[\frac{\Gamma\!\left(\frac{1-3s}2\right)}{\Gamma\!\left(\frac{3s-2}2\right)}
=\left(\frac{3t}2\right)^{3/2-3\sigma}\,\exp\!\left(-3it\log\frac{3t}2+3it+\frac{3\pi i}4\right)\left(1+O(t^{-1})\right),\]
substituting this back to the last integral, and forgetting the $O(t^{-1})$-term with an error $\ll k^{3/2}\,x^\varepsilon\,T^{1/2}$, it takes the form
\begin{equation*}
2\,(2\pi)^{-3/2}\,k\left(\frac{2^3\pi^3d^2mx}{k^3}\right)^{-\delta}
\Re\int\limits_\Lambda^T
t^{1/2+3\delta}\,\exp\!\left(
it\log\frac{2^3\pi^3d^2mx}{k^3t^3}
+3it+\frac{\pi i}4\right)
\mathrm dt.
\end{equation*}

The derivative of the phase is
\[\log\frac{2^3\,\pi^3\,d^2\,m\,x}{k^3}-3\log t,\]
and so the integrand has a unique saddle point at $t=(2\pi\,d^{2/3}\,(mx)^{1/3})/k$.
We will choose $T$ to be
\[T=\frac{2\pi\left(\!\left(N_k+\frac12\right)x\right)^{1/3}}k,\]
so that we have no saddle-points for $d^2\,m>N_k$. Furthermore, since $T\asymp(Nx)^{1/3}\,k^{-1}$, the errors incurred so far are
\begin{align*}
&\ll x^{1+\vartheta+\varepsilon}\,T^{-1}
+k^{3/2}\,x^\varepsilon\,T^{1/2}
\ll k\,x^{2/3+\varepsilon+\vartheta}\,N^{-1/3}+k\,x^{1/6+\varepsilon}\,N^{1/6}.
\end{align*}

Now we can treat the terms with $d^2m>N_k$. In the corresponding integrals the derivative of the phase is
\[\log\frac{2^3\pi^3d^2mx}{k^3t^3}\geqslant\log\frac{2^3\pi^3d^2mx}{k^3T^3}
=\log\frac{d^2m}{N_k+\frac12},\]
and so, by the first derivative test, and using the fact that $\log x\gg x-1$ for $x\in\left]1,2\right[$, they contribute
\begin{align*}
&\ll
k^{3/2+4\delta}\,x^{-\delta}\,T^{1/2+3\delta}\sum_{d\mid k}\frac1{d^{1+2\delta}}\sum_{d^2m>N_k}\frac{\left|A(d,m)\right|}{m^{1+\delta}}\cdot\frac1{\log\frac{d^2m}{N_k+\frac12}}\\
&\ll
k^{3/2+4\delta}\,
x^{-\delta}\,T^{1/2+3\delta}\sum_{d\mid k}\frac1{d^{1+2\delta}}\sum_{N_k<d^2m\leqslant2N_k}\frac{\left|A(d,m)\right|}{m^{1+\delta}\left(\frac{d^2m}{N_k+\frac12}-1\right)}
\\&\qquad+k^{3/2+5\delta}\,x^{-\delta}\,T^{1/2+3\delta}\\
&\ll k\,x^{1/6+\varepsilon}\,N^{1/6+\vartheta}+k\,x^{1/6+\varepsilon}\,N^{1/6},
\end{align*}
where we estimate the last sum by absolute values
\begin{align*}
&\sum_{d\mid k}\frac1{d^{1+2\delta}}\sum_{N_k<d^2m\leqslant2N_k}\frac{\left|A(d,m)\right|}{m^{1+\delta}\left(\frac{d^2m}{N_k+\frac12}-1\right)}\\
&\qquad\ll\sum_{d\mid k}d^{-1-2\delta}\left(\frac N{d^2}\right)^{-1-\delta}N\left(d\cdot\frac N{d^2}\right)^{\vartheta+\delta/2}\sum_{N_k<d^2m\leqslant2N_k}\frac1{d^2m-N_k-\frac12}\\
&\qquad\ll
\sum_{d\mid k}d^{1-\vartheta-5\delta/2}\,N^{-\delta/2+\vartheta}\,d^{-1}\left(1+\log^2k\right)\left(1+\log N\right)
\ll k^\delta\, N^\vartheta.
\end{align*}

In the remaining terms with $d^2m\leqslant N_k$, first extend each integral, obtained after applying the functional equation and Stirling's formula to simplify the $\Gamma$-factors, to be over the entire line segment connecting $-\delta-iT$ to $-\delta+iT$. This can be done with the error $\ll k^{3/2+\varepsilon}$. Also, we replace the factor $s^{-1}$ by the factor $(s+\Lambda)^{-1}$, which can be done with the error $\ll k^{3/2}\,x^\varepsilon\,T^{1/2}$. Thus, the remaining terms are
\begin{multline*}
k\,i^{-j}\sum_{d\mid k}\frac1d\sum_{d^2m\leqslant N_k}\frac{A(d,m)}m
\left(S\!\left(\overline h,m;\frac kd\right)+\left(-1\right)^jS\!\left(\overline h,-m;\frac kd\right)\right)\\
\cdot\frac1{2\pi i}\int\limits_{-\delta-iT}^{-\delta+iT}
\frac{\Gamma\!\left(\frac{1-3s}2\right)}{\Gamma\!\left(\frac{3s-2}2\right)}
\left(\frac{3^3\pi^3d^2mx}{k^3}\right)^s\frac{\mathrm ds}{s+\Lambda}.
\end{multline*}
Now, applying Corollary \ref{corollary-on-gl3-integrals}, we get the main terms
\begin{multline*}
\frac1{\pi\,\sqrt3}\,x^{1/3}\sum_{d\mid k}\frac1d\sum_{d^2m\leqslant N_k}
\frac{A(d,m)}{m^{2/3}}\\\cdot\left(S\!\left(\overline h,m;\frac kd\right)+(-1)^j\,S\!\left(\overline h,-m;\frac kd\right)\!\right)\cos\left(\frac{6\pi d^{2/3}(mx)^{1/3}}k\right).
\end{multline*}
The error terms $O(1)$ and $O(T^{1/2+3\delta/2})$ in Corollary \ref{corollary-on-gl3-integrals} are easily estimated to contribute $\ll k^{3/2}\,x^\varepsilon$ and $\ll k^{3/2}\,x^\varepsilon\,T^{1/2}$, respectively. Thus, it only remains to estimate the contribution from the last error term $O(T^{1/2}\,\log^{-1}(3^3T^3/2^3y))$. This is
\begin{align*}
&\ll k\sum_{d\mid k}\frac1d\sum_{d^2m\leqslant N_k}\frac{\left|A(d,m)\right|}m\left(\frac kd\right)^{1/2+\delta}\frac{T^{1/2}}{\log\frac{N_k+\frac12}{d^2m}}\\
&\ll k^{3/2+\varepsilon}\sum_{d\mid k}d^{-3/2-\varepsilon}\sum_{d^2m\leqslant N_k}\frac{(dm)^{\vartheta+\varepsilon}}m\,\frac{d^2\,m\,T^{1/2}}{N_k+\frac12-d^2m}\\
&\ll k^{3/2+\delta}\,T^{1/2}\,N^{\vartheta+\delta}\sum_{d\mid k}d^{1/2-\vartheta-\delta}\sum_{d^2m\leqslant N_k}\frac{1}{N_k+\frac12-d^2m}\\
&\ll k^{3/2+\delta}\,T^{1/2}\,N^{\vartheta+\delta}
\sum_{d\mid k}d^{-1/2-\vartheta-\delta}\left(1+\log^2k\right)\log N
\ll k\,x^{1/6+\varepsilon}\,N^{1/6+\vartheta},
\end{align*}
and we are done apart from a simple averaging over $j\in\left\{0,1\right\}$.
\end{proof}

\bigbreak\noindent\textbf{Remark.}
As a small clarification, let us explain, where the conditions for $N_k$ are really needed. After applying Perron's formula and the functional equation of the $L$-function, we get a series of integrals for each factor $d\mid k$. For high frequencies the integrands are oscillating, whereas for low frequencies the integrals possess saddle-points. As usual, the transition between the two is slightly tricky, and we only need the conditions for $N_k$ in two integrals for each $d\mid k$, namely the highest-frequency low-frequency term and the lowest-frequency high-frequency term.

\section{Proof of the second moment result in Theorem \ref{meansquare}}

\begin{proof}[Proof of Theorem \ref{meansquare}.]
By making a dyadic split, it is enough to consider the integral over the interval $[X,2X]$ instead of $[1,X]$. 
For $k\gg X^{2/3}$ we estimate trivially by using absolute values and get that
\begin{align*}
\int\limits_X^{2X}\left|\sum_{m\leqslant x}A(m,1)\,e\!\left(\frac{mh}k\right)\right|^2\mathrm d x\ll X^3\ll k^2\,X^{5/3}\ll k^2\,X^{5/3+2\vartheta+\varepsilon}.
\end{align*}

Suppose now that $k\ll X^{2/3}$, where the implicit constant is the one appearing in the condition $k\ll(xN)^{1/3}$ of Theorem \ref{gl3-truncated}. Clearly the integral is $\ll1$ if $X\in\left[1,2\right]$, and so we assume that $X\in\left[2,\infty\right[$.
Choose $N=X$ in the truncated Voronoi identity and denote
\begin{align*}
M\!\left(x,\frac hk\right):=\frac{x^{1/3}}{\pi\sqrt 3}\sum_{d|k}\frac 1d\sum_{d^2m\leqslant X_k}\frac{A(d,m)}{m^{2/3}}\,S\!\left(\overline h,m;\frac kd\right)\cos\!\left(\frac{6\pi d^{2/3}(mx)^{1/3}}k\right).
\end{align*}
Then the truncated Voronoi identity takes the form
\begin{equation*}
\sum_{m\leqslant x}A(m,1)\,e\!\left(\frac{mh}k\right)=M\!\left(x,\frac hk\right)+O\!\left(k\,x^{2/3+\vartheta+\varepsilon}\,X^{-1/3}+k\,x^{1/6+\varepsilon}\,X^{1/6+\vartheta}\right).
\end{equation*}

\noindent The error terms of the truncated Voronoi identity are easily seen to contribute $\ll k^2\,X^{5/3+2\vartheta+\varepsilon}$.
The following lemma treats the contribution coming from the absolute square of the main term $M(x,h/k)$.
\end{proof}

\begin{lemma}\label{mainterm} For any $X\in\left[2,\infty\right[$ and coprime integers $h$ and $k$ with $1\leqslant k\leqslant X$, we have
\begin{align*}
\int\limits_X^{2X}\left|M\!\left(x,\frac hk\right)\right|^2\mathrm d x\ll k^2\,X^{5/3+\vartheta+\varepsilon}.
\end{align*}
\end{lemma}

\begin{proof}
The sum $M(x,h/k)$ can be considered a sum $\Sigma_{d,m}$. The absolute value square $\left|M(x,h/k)\right|^2$ can be considered as a sum $\Sigma_{d_1,m_1}\overline{\Sigma_{d_2,m_2}}$. We integrate the double sum termwise and consider separately the ``diagonal terms'' in which $d_1^2m_1=d_2^2m_2$ and the ``off-diagonal terms'' in which $d_1^2m_1\neq d_2^2m_2$.

The diagonal terms contribute, using Weil's estimate for the Kloosterman sums, estimating the integrals by absolute values, applying the Cauchy--Schwarz inequality to the double sum over $m_1$ and $m_2$, where each possible value for $m_1$ uniquely determines the value of $m_2$, and vice versa, and finally employing the Rankin--Selberg estimate for the Fourier coefficients,
\begin{align*}
&\ll\sum_{d_1\mid k}\sum_{d_2\mid k}\underset{d_1^2m_1=d_2^2m_2}{\sum_{d_1^2m_1\leqslant X_k}\sum_{d_2^2m_2\leqslant X_k}}\frac{\left|A(d_1,m_1)\,A(d_2,m_2)\right|}{d_1\,d_2\,m_1^{2/3}\,m_2^{2/3}}\left(\frac{k^2}{d_1\,d_2}\right)^{1/2+\varepsilon}\,X^{5/3}\\
&\ll k^{1+\varepsilon}\,X^{5/3}\sum_{d_1\mid k}d_1^{-3/2-\varepsilon}\sum_{d_2\mid k}d_2^{-3/2-\varepsilon}\\
&\qquad\cdot
\sqrt{\underset{d_1^2m_1=d_2^2m_2}{\sum_{d_1^2m_1\leqslant X_k}\sum_{d_2^2m_2\leqslant X_k}}\frac{\left|A(d_1,m_1)\right|^2}{m_1^{4/3}}}
\sqrt{\underset{d_1^2m_1=d_2^2m_2}{\sum_{d_1^2m_1\leqslant X_k}\sum_{d_2^2m_2\leqslant X_k}}\frac{\left|A(d_2,m_2)\right|^2}{m_2^{4/3}}}\\
&\ll k^{1+\varepsilon}\,X^{5/3}\sum_{d_1\mid k}d_1^{-3/2-\varepsilon}\sum_{d_2\mid k}d_2^{-3/2-\varepsilon}\sqrt{d_1\,d_2}
\ll k^{1+\varepsilon}\,X^{5/3},
\end{align*}
as required.

In the off-diagonal terms, we meet integrals of the form
\[\int\limits_X^{2X}x^{2/3}\,\cos\!\left(\frac{6\pi d_1^{2/3}m_1^{1/3}x^{1/3}}k\right)\cos\!\left(\frac{6\pi d_2^{2/3}m_2^{1/3}x^{1/3}}k\right)\mathrm dx,\]
and by the first derivative test these are
\[\ll\frac{k\,X^{4/3}}{\left|d_1^{2/3}\,m_1^{1/3}-d_2^{2/3}\,m_2^{1/3}\right|}.\]
Thus, the off-diagonal terms contribute, by using Weil's estimate and Rankin-Selberg estimate,
\begin{align*}
&\ll\sum_{d_1\mid k}\sum_{d_2\mid k}\underset{d_1^2m_1<d_2^2m_2}{\sum_{d_1^2m\leqslant X_k}\sum_{d_2^2m\leqslant X_k}}\frac{\left|A(d_1,m_1)\,A(d_2,m_2)\right|}{d_1\,d_2\,m_1^{2/3}\,m_2^{2/3}}\\&\qquad\cdot\left(\frac{k^2}{d_1d_2}\right)^{1/2+\varepsilon}\frac{k\,X^{4/3}\,(d_2^{2/3}\,m_2^{1/3})^2}{\left|d_1^2\,m_1-d_2^2\,m_2\right|}\\
&\ll k^{2+\varepsilon}\,X^{4/3}\,\sum_{d_1\mid k}d_1^{-3/2-\varepsilon}\sum_{d_2\mid k}d_2^{-3/2+4/3-\varepsilon}\\&\qquad\cdot\sum_{d_1^2m_1\leqslant X_k}\frac{\left|A(d_1,m_1)\right|}{m_1^{2/3}}\sum_{d_1^2m_1<d_2^2m_2\leqslant X_k}\frac{\left|A(d_2,m_2)\right|}{d_2^2m_2-d_1^2m_1}\\
&\ll k^{2+\varepsilon}\,X^{4/3}\,\sum_{d_1\mid k}d_1^{-3/2-\varepsilon}\sum_{d_2\mid k}d_2^{-3/2+4/3-\varepsilon}\sum_{d_1^2m_1\leqslant X_k}\frac{\left|A(d_1,m_1)\right|}{m_1^{2/3}}\,d_2^{\vartheta+\varepsilon}\left(\frac X{d_2^2}\right)^{\vartheta+\varepsilon}\\
&\ll k^{2+\varepsilon}\,X^{4/3}\,\sum_{d_1\mid k}d_1^{-3/2-\varepsilon}\sum_{d_2\mid k}d_2^{-3/2+4/3-\varepsilon}d_1^{1/3}\,X^{1/3}\,d_2^{-\vartheta+\varepsilon}\,X^{\vartheta+\varepsilon}\\
&\ll k^{2+\varepsilon}\,X^{5/3+\vartheta+\varepsilon}\sum_{d_1\mid k}d_1^{-7/6-\varepsilon}\sum_{d_2\mid k}d_2^{-1/6-\vartheta-\varepsilon}
\ll k^2\,X^{5/3+\vartheta+\varepsilon},
\end{align*}
as required.
\end{proof}

\section{Proof of the pointwise bound in Corollary \ref{pointwise-gl3}}

\begin{proof}[Proof of Corollary \ref{pointwise-gl3}.]
We may assume that $k\ll x^{2/3}$ with a sufficiently small implicit constant so that Theorem \ref{gl3-truncated} may be applied, as for larger $k$ it is enough to observe that the Rankin--Selberg estimate says that the sum is $\ll x$.
The main terms from the truncated Voronoi identity contribute
\begin{align*}
&\ll x^{1/3}\sum_{d\mid k}\frac1d\sum_{d^2m\leqslant N_k}\frac{\left|A(d,m)\right|}{m^{2/3}}\left(\frac kd\right)^{1/2+\varepsilon}\\
&\ll k^{1/2+\varepsilon}\,x^{1/3}\sum_{d\mid k}d^{-3/2}\sum_{d^2m\leqslant N_k}\frac{\left|A(d,m)\right|}{m^{2/3}}\\
&\ll k^{1/2+\varepsilon}\,x^{1/3}\sum_{\substack{d\mid k,\\d\leqslant\sqrt{2N}}}d^{-3/2}\cdot d\cdot\left(\frac N{d^2}\right)^{1/3}\\
&\ll k^{1/2+\varepsilon}\,x^{1/3}\,N^{1/3}\sum_{\substack{d\mid k,\\d\leqslant\sqrt{2N}}}d^{-7/6}\ll k^{1/2+\varepsilon}\,x^{1/3}\,N^{1/3}.
\end{align*}
In particular, the sum over $d^2m\leqslant N_k$ vanishes identically if $d>\sqrt{2N}$.

Choosing $N=x$ gives
\[\ll k^{1/2+\varepsilon}\,x^{2/3}+k\,x^{1/3+\vartheta+\varepsilon}
.\]

When $k\ll x^{2/3-2\vartheta}$, we may choose $N=k^{3/4}\,x^{1/2+3\vartheta/2}$, which satisfies $1\leqslant N\ll x$ and $k\ll(xN)^{1/3}$, the latter with the implicit constant required by Theorem \ref{gl3-truncated}, to get
\[\ll k^{3/4}\,x^{1/2+\vartheta/2+\varepsilon}+k^{9/8+3\vartheta/4}\,x^{1/4+3\vartheta^2/2+3\vartheta/4+\varepsilon},\]
as claimed.
\end{proof}

\section*{Acknowledgements}

The authors would like to express their sincere gratitude to Anne-Maria Ernvall-Hyt\"onen for useful discussions and invaluable support.

During this research, the first author was funded by the Doctoral Programme for Mathematics and Statistics of the University of Helsinki. The second author was funded by the Academy of Finland through the Finnish Centre of Excellence in Inverse Problems Research and the projects 283262, 276031 and 282938, and by the Basque Government through the BERC 2014--2017 program and by Spanish  Ministry of Economy and Competitiveness MINECO: BCAM  Severo Ochoa excellence accreditation SEV-2013-0323.

\end{document}